\documentclass[12pt]{article}%

\usepackage{amsmath,enumerate}
\usepackage{amsfonts}
\usepackage{amssymb}

\newtheorem{theorem}{Theorem}[section]

\newtheorem{conjecture}[theorem]{Conjecture}
\newtheorem{corollary}[theorem]{Corollary}

\newtheorem{lemma}[theorem]{Lemma}

\newenvironment{proof}[1][Proof]{\noindent\textbf{#1.} }
{\hfill \ \rule{0.5em}{0.5em}}

\newcommand{\integers}{\mathbb{Z}}

\begin{document}

\title{$k$-fold Sidon sets}

\author{Javier Cilleruelo\thanks{
Departamento de Matem\'{a}ticas, Universidad Aut\'{o}noma de Madrid, 28049 Madrid.}
\and
Craig Timmons\thanks{Department of Mathematics, University of California San Diego, La Jolla, CA 92093.
Partially supported by NSF Grant DMS-1101489 through Jacques Verstra\" ete.}}

\maketitle
\parskip 1.5mm
\begin{abstract}
Let $k \geq 1$ be an integer.  A set $A \subset \mathbb{Z}$ is a \emph{$k$-fold Sidon set} if $A$ has only trivial solutions to each equation of the form $c_1 x_1 + c_2 x_2 + c_3 x_3 + c_4 x_4 = 0$ where
$0 \leq |c_i | \leq k$, and $c_1 + c_2 + c_3 + c_4 = 0$.  We prove that for any integer $k \geq 1$, a $k$-fold Sidon
set $A \subset [N]$ has at most $(N/k)^{1/2} + O((Nk)^{1/4})$ elements. Indeed we prove that given any $k$ positive integers $c_1<\cdots <c_k$, any set $A\subset [N]$ that contains only trivial solutions
to $c_i(x_1-x_2)=c_j(x_3-x_4)$ for each $1 \le i \le j \le k$, has at most $(N/k)^{1/2}+O((c_k^2N/k)^{1/4})$ elements. On the other hand, for any $k \geq 2$ we can exhibit $k$ positive integers $c_1,\dots, c_k$ and a set $A\subset [N]$ with
$|A|\ge (\frac 1k+o(1))N^{1/2}$, such that $A$ has only trivial solutions to $c_i(x_1 - x_2) = c_j (x_3 -  x_4)$
for each $1 \le i \le j\le k$.
\end{abstract}


\section{Introduction}

Let $\Gamma$ be an abelian group.  A set $A \subset \Gamma$ is a \emph{Sidon} set if $a + b = c + d$ and $a,b,c,d \in A$ implies $\{a,b \} = \{c,d \}$.  Sidon sets in $\integers$ and in the group $\integers_N:= \integers / N \integers$ have been studied extensively.  Erd\H{o}s and Tur\'{a}n \cite{ET} proved that a Sidon set $A \subset [N]$ has at most $N^{1/2} + O(N^{1/4})$ elements.  Constructions of
Singer \cite{S}, Bose and Chowla \cite{BC}, and Ruzsa \cite{R} show that this upper bound is asymptotically best possible.  It is a prize problem of Erd\H{o}s \cite{E} to determine whether or not the error term is bounded.  For more on Sidon sets we recommend O'Bryant's survey \cite{OB}.

Let
\begin{equation}\label{r eq}
c_1 x_1 + \dots + c_r x_r = 0
\end{equation}
be an integer equation where $c_i \in \integers \backslash \{ 0 \}$, and $c_1 + \dots + c_r = 0$.
Call such an equation an \emph{invariant equation}.
A solution $(x_1 , \dots , x_r ) \in \integers^r$ to (\ref{r eq}) is \emph{trivial} if there is a partition of
$\{1, \dots , r \}$ into nonempty sets $T_1 , \dots , T_m$ such that for every $1 \leq i \leq m$, we have
$\sum_{j \in T_i} c_j = 0$, and $x_{j_1} = x_{j_2}$ whenever $j_1 , j_2 \in T_i$.  A natural extremal problem is to determine the maximum size of a set $A \subset [N]$ with only trivial solutions to (\ref{r eq}).
This problem was investigated in detail by Ruzsa \cite{R}.
One of the important open problems from \cite{R} is the genus problem.
Given an invariant equation $E : c_1 x_1+ \dots + c_r x_r = 0$, the \emph{genus} $g(E)$ is the largest
integer $m$ such that there is a partition of $\{1, \dots , r \}$ into nonempty sets $T_1 , \dots , T_m$, such that $\sum_{j \in T_i} c_j = 0$ for $1 \leq i \leq m$.  Ruzsa proved that if $E$ is an invariant equation
and $A \subset [N]$ has only trivial solutions to $E$, then $|A| \leq c_E N^{1 / g(E)}$.  Here $c_E$ is a positive constant
depending only on the equation $E$.
Determining if there are sets $A \subset [N]$ with $|A| = N^{1 / g(E) - o(1)}$ and having only trivial solutions to $E$ is open for most equations.  In particular, the genus problem is open for the equation $2x_1 + 2x_2 = 3x_3 + x_4$.  This equation has genus 1 but the best known construction \cite{R} gives a set $A \subset [N]$ with $|A| \geq c N^{1/2}$ where $c>0$ is a positive constant.
More generally, Ruzsa showed that for any four variable equation $E:c_1 x_1 + c_2 x_2 = c_3 x_3  + c_4 x_4$ with
$c_1 + c_2  = c_3 + c_4$ and $c_i \in \mathbb{N}$, there is a set $A \subset [N]$ with only trivial solutions
to $E$ and $|A| \geq c_E N^{1/2 - o(1)}$.  In this paper we consider special types of four variable invariant equations.

Let $k \geq 1$ be an integer.
A set $A \subset \integers$ is a \emph{$k$-fold Sidon} set if $A$ has only trivial solutions to each equation of the form
\begin{equation*}
c_1 x_1 + c_2 x_2 + c_3 x_3 + c_4 x_4 = 0
\end{equation*}
where $0 \leq |c_i | \leq k$, and $c_1 + c_2 +c_3 + c_4 = 0$.  A 1-fold Sidon set is a Sidon set.  A 2-fold Sidon set has only trivial solutions to each of the equations
\begin{equation*}
x_1 + x_2 - x_3 - x_4 = 0, ~~~~~ 2x_1 + x_2 - 2x_3 - x_4 = 0, ~~~~~ 2x_1 - x_2 - x_3 = 0.
\end{equation*}
One can also define $k$-fold Sidon sets in $\integers_N$.  We must add the condition that $N$ is relatively prime to
all integers in the set $\{1,2, \dots , k \}$.  The reason for this is that if a coefficient $c_i \in \{1,2, \dots , k \}$ has a common factor with $N$,
then in $\mathbb{Z}_{N}$ one could have $c_i ( a_1 - a_2 ) = 0$ with $a_1 \neq a_2$.  In this case, if $|A| \geq 3$, we can choose $a_3 \in A \backslash \{a_1,a_2 \}$, and obtain the nontrivial solution $(x_1,x_2,x_3,x_4) = (a_1,a_2,a_3,a_3)$
to the equation $c_i ( x_1 - x_2 ) + x_3 - x_4 = 0$.

Lazebnik and Verstra\"{e}te \cite{LV} were the first to define $k$-fold Sidon sets.  They conjectured the following.

\begin{conjecture}[Lazebnik, Verstra\"{e}te \cite{LV}]\label{LV conjecture}
For any integer $k \geq 3$, there is a positive constant $c_k > 0$ such that for all integers $N \geq 1$, there is a
$k$-fold Sidon set $A \subset [N]$ with $|A| \geq c_k N^{1/2}$.
\end{conjecture}

This conjecture is still open.  Lazebnik and Verstra\"{e}te proved that for infinitely many $N$, there is a 2-fold Sidon set $A \subset \integers_N$ with $|A| \geq \frac{1}{2}N^{1/2} - 3$.  Axenovich \cite{A} and Verstra\"{e}te (unpublished) observed that
one can adapt Ruzsa's construction for four variable equations (Theorem 7.3, \cite{R}) to construct $k$-fold Sidon
sets $A \subset [N]$ or $A \subset \integers_N$ with $|A| \geq c_k N^{1/2} e^{ - c_k \sqrt{ \log N } }$ for any $k \geq 3$.
An affirmative answer to Conjecture~\ref{LV conjecture}, even in the case when $k=3$, would have applications to hypergraph Tur\'{a}n problems \cite{LV} and extremal graph theory \cite{TV}.

Since any $k$-fold Sidon set is a Sidon set, the trivial upper bound $|A|\le \sqrt{N-3/4}+1/2$ for a Sidon set
$A \subset \mathbb Z_N$, and the Erd\H{o}s-Tur\'{a}n bound  $|A| \leq N^{1/2} + O(N^{1/4})$ for any Sidon set $A \subset [N]$, also hold for  $k$-fold Sidon sets.
We will obtain better upper bounds for $k$-fold Sidon sets.  Instead of considering all the possible equations $c_1x_1+c_2x_2+c_3x_3+c_4x_4=0$ with $c_1+c_2+c_3+c_4=0$, we will take advantage only of the equations of the form
$$c_1(x_1-x_2)=c_2(x_3-x_4).$$

For any $c_1,\dots,c_k$ with $(c_i,N)=1$, if $A\subset \mathbb Z_N$ contains only trivial solutions
to $c_i(x_1-x_2)=c_j(x_3-x_4)$ for each $1 \le i \le j \le k$, then
\begin{equation}\label{group}
|A|\le \sqrt{ \frac{N-1}k+\frac 14 }+\frac 12.
\end{equation}
To see this, consider all elements of the form $c_i(x-y)$ where $1 \le i \le k$, and $x \neq y$ are elements of $A$.
All of these elements are distinct and nonzero.  Therefore, $k|A|(|A|-1)\le N-1$ which is equivalent to (\ref{group}).

The short counting argument used to obtain \eqref{group} does not work in $\integers$.
Using a more sophisticated argument, we can show that a bound similar to \eqref{group} does hold in $\mathbb{Z}$.

\begin{theorem}\label{k ub}
Let $k \geq 1$ be an integer and $1 \leq c_1 < c_2 <  \dots < c_k$ be a set of $k$ distinct integers.  If $A \subset [N]$ is a set with only trivial solutions to $c_i (x_1 - x_2)  = c_j (x_3 - x_4)$ for each $1 \leq i \leq j \leq k$, then
\begin{equation*}
|A| \leq \left( \frac{N}{k} \right)^{1/2} + O \left( \left(  \frac{c_k^2N}{k} \right)^{1/4} \right).
\end{equation*}
\end{theorem}
Taking $c_j=j$ for $1\le j\le k$, we have the following corollary.
\begin{corollary}\label{k cor}
If $k \geq 1$ is an integer and $A \subset [N]$ is a $k$-fold Sidon set, then
\[
|A| \leq \left( \frac{N}{k} \right)^{1/2} + O ( (kN)^{1/4} ).
\]
\end{corollary}

It is natural to ask if we can improve Corollary~\ref{k cor} if we make full use of the assumption that $A$ is a $k$-fold Sidon set.
For example, the bound $|A| \leq (N/3)^{1/2} + O(N^{1/4})$ holds under the assumption that
$A \subset [N]$ has only trivial solutions to $c_1 (x_1 - x_2) = c_2 (x_3 - x_4)$ for each $1 \leq c_1 \leq c_2 \leq 3$.  A 3-fold Sidon set additionally has only trivial solutions to $2x_1 + 2x_2 = 3x_3 + x_4$.  Our argument does not capture this property.  It is not known if this additional assumption would improve the upper bound $|A| \leq (N/3)^{1/2} + O(N^{1/4})$.

The method used by Lazebnik and Verstra\"{e}te to construct 2-fold Sidon sets is rather robust.
Using this method, we prove the following theorem.

\begin{theorem}\label{weak construction}
There exist $k$ distinct integers $c_1,\dots,c_k$ and  infinitely many $N$, such that there is a set $A \subset \integers_{N}$ with
\[
|A| \geq \frac{N^{1/2}}{k}(1 - o(1) )
\]
and  having only trivial solutions to $c_i(x_1 - x_2) = c_j(x_3 - x_4)$ for each $1\le i\le j\le k$.
\end{theorem}

The next section contains the proof of Theorem~\ref{k ub}.  Section 3 contains the proof of
Theorem~\ref{weak construction}.


\section{Proof of  Theorem~1.2}


 For finite sets $B , C \subset \mathbb{Z}$, define
\[
r_{B - C}(x) = |  \{ (b,c) : b - c = x , b \in B , c \in C \} |.
\]
The following useful lemma has appeared in the literature (see \cite{C} or \cite{R}).
\begin{lemma}\label{lemma1}
For any finite sets $B , C \subset \mathbb{Z}$,
\begin{equation}\label{eq lemma1}
\frac{  ( |B| |C| )^2 }{ | B + C | } \leq |B | |C| + \sum_{x \neq 0} r_{B - B} (x) r_{C - C}(x).
\end{equation}
\end{lemma}
\begin{proof}
By Cauchy-Schwarz,
\begin{eqnarray*}
\frac{  ( |B| |C| )^2 }{ | B + C | } & = & \frac{  \left(  \sum_{x \in B + C} r_{B + C}(x) \right)^2 }{ |B + C| }
\leq  \sum_x r_{B + C}^2 (x) \\
& = & \sum_x r_{B - B} (x) r_{C  - C} (x) = |B| |C| + \sum_{x \neq 0 } r_{B - B}(x) r_{C - C}(x).
\end{eqnarray*}
\end{proof}

\begin{proof}[Proof of Theorem~\ref{k ub}]  Let $1 \leq c_1 < c_2 < \dots < c_k$ be $k$ distinct integers.
Let $A \subset [N]$ be a set with only trivial solutions to $c_i (x_1 - x_2) = c_j (x_3 - x_4)$ for each $1 \leq i \leq j \leq k$.
Let
\[
B_{r,i} = \{ x : c_r x + i \in A \}
\]
for $1 \leq r \leq k$ and $0 \leq i \leq c_r -1$.  
 Therefore,
 $$|A| =\sum_{i=0}^{c_r - 1} |\{a\in A:\ a\equiv i\pmod{c_r}\} |= \sum_{i=0}^{c_r - 1} | B_{r,i} |$$ so by Cauchy-Schwarz,
\begin{equation}\label{eq 2}
|A|^2 = \left(  \sum_{i=0}^{c_r-1} |B_{r,i} | \right)^2 \leq c_r \sum_{i=0}^{c_r-1} |B_{r,i} |^2.
\end{equation}
For any $y \neq 0$,
\begin{equation}\label{eq 1.5}
\sum_{r = 1}^{k} \sum_{i=0}^{c_r - 1} r_{B_{r,i} - B_{r,i} } (y) \leq 1.
\end{equation}
To see this, suppose
\begin{equation}\label{eq 2.5}
y = x_1 - x_2 = x_3 - x_4
\end{equation}
where $x_1 , x_2 \in B_{r,i}$ and $x_3 , x_4 \in B_{r' , i'}$ for some
$1 \leq r, r' \leq k$, $1 \leq i \leq c_r - 1$, and $1 \leq i' \leq c_{r'} - 1$.  There are elements $a_1 , a_2 , a_3 , a_4 \in A$ such that
\[
c_r x_1 + i = a_1,~~ c_r x_2 + i = a_2,~~ c_{r'} x_3 + i' = a_3, ~~\mbox{and}~~c_{r'} x_4 + i' = a_4.
\]
Then (\ref{eq 2.5}) implies
\[
\frac{1}{c_r} (a_1 - i ) - \frac{1}{c_r} ( a_2 - i) = \frac{1}{c_{r'}} ( a_3 - i') - \frac{1}{ c_{r'} } ( a_4 - i'),
\]
thus $c_{r'} (a_1 - a_2) = c_r (a_3 - a_4)$.  Since $y \neq 0$, we have $a_1 \neq a_2$ and $a_3 \neq a_4$ and the we would have a non trivial solution of the equation.

Let $C = \{ 0 ,1, \dots , m-1 \}$.  For any $1 \leq r \leq k$ and $0 \leq i \leq c_r-1$,
the set $B_{r,i} + C$ is contained in the interval $ \{ 0 ,1 , \dots , N/c_r + m - 1 \}$.
This gives the trivial estimate $| B_{r , i } + C | \leq N/c_r  + m$.  By Lemma~\ref{lemma1},
\[
\frac{ |B_{r,i}|^2 m^2 }{ N/c_r + m} \leq |B_{r,i} | m + \sum_{y \neq 0} r_{B_{r,i} - B_{r,i} } (y) r_{C - C}(y).
\]
We sum this inequality over all $1 \leq r \leq k$ and $0 \leq i \leq c_r-1$ to get
\begin{eqnarray*}
m^2 \sum_{r =1 }^{k} \frac{1}{N/c_r  + m } \sum_{i=0}^{c_r-1} |B_{r,i} |^2 & \leq & \sum_{r = 1}^{k} \sum_{i=0}^{c_r-1} |B_{r,i} | m \\
& + & \sum_{y \neq 0} \sum_{r = 1}^{k} \sum_{i=0}^{c_r-1} r_{B_{r,i}  - B_{r,i} }(y)  r_{C - C}(y) \\
& \leq & k |A|m + \sum_{y \neq 0} r_{C - C}(y) \\
& \leq & m( k |A| + m ).
\end{eqnarray*}
From (\ref{eq 2}) we deduce
\begin{equation}\label{eq 3}
m^2 |A|^2 \sum_{r =1}^{k} \frac{1}{ N + c_r m} \leq m( k|A| + m).
\end{equation}
The left hand side of (\ref{eq 3}) is at least $\frac{ |A|^2 km^2 }{N + c_k m}$.  Therefore,
$\frac{ |A|^2 km }{N + c_k m} \leq k |A| + m$, and
\[
|A|^2 km \leq ( N + c_k m)( m + k |A| ).
\]
From this inequality, we obtain
\begin{eqnarray*}
\left( |A| - \left( \frac{N}{2m} + \frac{c_k}{2} \right) \right)^2 & \leq & \frac{N}{k} + \frac{c_k m}{k} + \left( \frac{N}{2m} + \frac{c_k}{2} \right)^2 \\
& \leq & \frac{N}{k} + \frac{c_k m}{k} + \frac{N^2}{2m^2} + \frac{c_k^2}{2} \\
& = & \frac{N}{k} \left( 1 + \frac{c_k m}{N} + \frac{Nk}{2m^2} + \frac{k c_k^2}{2N} \right).
\end{eqnarray*}
Upon solving for $|A|$, we get
\begin{eqnarray*}
|A| & \leq & \left( \frac{N}{k} \right)^{1/2} \left(1 + \frac{c_k m}{N} + \frac{Nk}{2m^2} + \frac{k c_k^2}{2N} \right) + \frac{N}{2m} + \frac{c_k}{2} \\
& \leq & \left( \frac{N}{k} \right)^{1/2} + \frac{c_k m}{k^{1/2} N^{1/2}} + \frac{ N^{3/2} k^{1/2} }{2m^2}
+ \frac{k^{1/2}c_k^2 }{2N^{1/2} } + \frac{N}{2m} + \frac{c_k}{2}.
\end{eqnarray*}
Take $m = \lceil (N^{3/4} k^{1/4})/c_k^{1/2} \rceil$ to get $|A| \leq \left( \frac{N}{k} \right)^{1/2} + O ( (c_k^2 N/k)^{1/4})$.
This completes the proof of Theorem~\ref{k ub}.

\end{proof}


\section{Proof of Theorem~1.4}


Let $k \geq 2$ be an integer.  Let $p$ be a prime, and let $M \geq 1$ be a large integer.
Let $r$ be any prime with $r > Mk$.  Let $i \geq 1$ be an integer, and set $t = r^i$ and $q = p^t$.

We will prove that  for $c_j=p^{j-1}$ for $j=1,\dots k$ there exists a set $A\subset \mathbb Z_{q^2-1}$ with $|A|\ge \frac{q}{k} \left( 1 - \frac{1}{M} \right) - (p^4 - 1)(M - 1)$ and having only trivial solutions to
$$x_1-x_2=p^{j-1}(x_3-x_4)$$
for $1 \leq j \leq k$.  This proves Theorem~\ref{weak construction} because  as $i$ tends to infinity, the term
$\frac{q}{k} \left( 1 - \frac{1}{M} \right)$ is the dominant term.  $M$ can be taken  as large as we want, and $(p^4 - 1)(M-1)$ is constant with respect to $i$.

Let $\theta$ be a generator of the cyclic group $\mathbb{F}_{q^2}^{*}$.
Bose and Chowla \cite{BC} proved that the set
\begin{equation*}
C (q , \theta ) = \{ a \in \mathbb{Z}_{q^2 - 1} :  \theta^{a} - \theta \in \mathbb{F}_q \}
\end{equation*}
is a Sidon set in $\mathbb{Z}_{q^2 - 1}$.  Lindstr\"{o}m \cite{L} proved
\[
B(q , \theta ) = \{ b \in \mathbb{Z}_{q^2 - 1}:  \theta^b + \theta^{qb} = 1 \}
\]
is a translate of $C(q , \theta)$ and is therefore a Sidon set.

\begin{lemma}\label{closure}
The map $x \mapsto px$ is an injection from $\mathbb{Z}_{q^2- 1}$ to $\mathbb{Z}_{q^2 - 1}$ that maps $B (q , \theta )$ to $B(q, \theta )$.
\end{lemma}
\begin{proof}
The map $x \mapsto px$ is 1-to-1 since $p$ is relatively prime to $q^2 -1$.  If $b \in B ( q , \theta)$, then
\[
1 = ( \theta^{b} + \theta^{qb} )^p = \theta^{pb} + \theta^{q(pb)}
\]
so $pb \in B(q , \theta)$.
\end{proof}

Let $\pi : B(q , \theta ) \rightarrow B ( q , \theta )$ be the permutation $\pi (b) = pb$.  As in \cite{LV}, we use the cycles of
$\pi$ to define $A$.  Let $\sigma = (b_1 , \dots , b_m)$ be a cycle of $\pi$.  If $m < k$, then remove all elements of $\sigma$ from $B(q,\theta)$.  If $m \geq k$, then remove all $b_j$ in $\sigma$ for which $j$ is not divisible by $k$.  Do this for each cycle
of $\pi$.  Let $A$ be the resulting subset of $B(q, \theta )$.

\begin{lemma}\label{claim 1}
For each $c \in \{1,p , p^2 , \dots , p^{k-1} \}$, $A$ has only trivial solutions to
\[
x_1 - x_2 = c( x_3 - x_4).
\]
\end{lemma}
\begin{proof}
Suppose $a_1 , a_2 , a_3 , a_4 \in A$ and $a_1 - a_2 = p^j ( a_3 - a_4 )$ for some $0 \leq j \leq k-1$.
By Lemma~\ref{closure}, there are elements $b_3 , b_4 \in B(q, \theta )$ such that $p^j a_3 = b_3$ and $p^j a_4 = b_4$.
This gives $a_1 - a_2 = b_3 - b_4$.  Since $B(q , \theta )$ is a Sidon set, either $a_1 =a_2$, $b_3 = b_4$ or
$a_1 = b_3$, $a_2 = b_4$.

If $a_1 = a_2$ and $b_3 = b_4$, then $a_3 = a_4$ and the solution $(a_1,a_2,a_3,a_4)$ is trivial.  Suppose
$a_1 = b_3$ and $a_2 = b_4$.  This implies $b_3 \in A$, so both $p^j a_3$ and $a_3$ are in $A$.  This contradicts the way in which $A$ was constructed.
\end{proof}

\begin{lemma}\label{claim 2}
$|A| \geq \frac{q}{k} \left( 1 - \frac{1}{M} \right) - (p^4 - 1) (M-1)$.
\end{lemma}
\begin{proof}
In order to obtain a lower bound on $|A|$, we need to estimate the number of cycles of $\pi$ that are short.  For instance, if
all cycles of $\pi$ have length less than $k$, then $|A| = 0$.  For a cycle $\sigma$ of $\pi$ with length $mk  \geq Mk$, we delete at most
$m(k-1)$ elements from $B(q, \theta )$ and keep at least $m-1$ elements.

We estimate the number of cycles of length at most $Mk-1$.  Let $\sigma = ( b , pb , \dots , p^{e-1} b)$ be a cycle of $\pi$ of length
$e$ where $e \leq Mk - 1$.  The integer $e$ is the smallest positive integer such that $p^e b \equiv b ( \textup{mod}~q^2 - 1 )$.  This is the same
as saying that the order of $p$ in the multiplicative group of units $\integers_{n}^{*}$ is $e$ where
$n = \frac{ q^2 - 1}{ \textup{gcd}( b , q^2 - 1) }$.  Since
\[
p^{4t} - 1 = (p^{2t} - 1)( p^{2t} + 1) = (q^2 - 1)(p^{2t} +1 )
\]
we have $p^{4t} \equiv 1 ( \textup{mod}~q^2 - 1)$, so $e$ must divide $4t = 4r^i$.  Since $r$ is prime and $r \geq Mk$, $e$ cannot divide $r$, so $e$ must divide 4.  To count the number the number of cycles of $\pi$ with length at most $Mk-1$, it is enough to count the elements $x \in \integers_{q^2 - 1} \backslash \{0 \}$ such that $p^4 x \equiv x (\textup{mod}~q^2 - 1)$.  This follows from the fact that if $e \in \{1,2 \}$ and $p^e x \equiv x ( \textup{mod}~q^2 - 1)$, then
$p^4 x \equiv x ( \textup{mod}~q^2 - 1)$.  The number of solutions to this congruence is
$\textup{gcd}(p^4 - 1 , q^2 - 1) \leq p^4 - 1$.  Therefore, there are at most $p^4 - 1$ cycles of $\pi$ of length at most $Mk-1$.
For a cycle of length at least $Mk$, the proportion of elements of the cycle that are put into $A$ is at least
$\frac{ M - 1}{Mk}$ (the function $f(x) = \frac{x-1}{xk}$ is increasing provided $k > 0$).
Since $| B(q , \theta ) | = q$,
\[
|A|  \geq  \left( q - ( p^4 - 1) Mk \right) \left( \frac{M-1}{Mk} \right) = \frac{q}{k} \left( 1 - \frac{1}{M} \right) - (p^4 - 1)(M - 1).
\]
\end{proof}

Theorem~\ref{weak construction} follows from Lemmas~\ref{claim 1} and \ref{claim 2}.


\section{Concluding Remarks}

The most important open problem concerning $k$-fold Sidon sets is an answer to Conjecture~\ref{LV conjecture}.
The case $k=3$ is particularly interesting.  A 3-fold Sidon set $A \subset [N]$ with $|A| \geq cN^{1/2}$ is known to imply the existence of a graph with $c_1 N$ vertices, $c_2 N^{3/2}$ edges, and every edge is in exactly one cycle of length four \cite{TV}.

 Another problem is to determine the maximum size of a 2-fold Sidon set in $\mathbb{Z}_N$ or $[N]$.  Let
$S_k(N)$ be the maximum size of a $k$-fold Sidon set in $\mathbb{Z}_N$.
For any integer $t \geq 1$, there are 2-fold Sidon sets $A \subset \mathbb{Z}_N$, $N = 2^{2^{t+1}} + 2^{2^t} + 1$, with $|A| \geq \frac{1}{2}N^{1/2} - 3$ (see \cite{LV}).  Theorem~\ref{k ub} gives an upper bound of $(N/2)^{1/2} + O(N^{1/4})$ so
\[
\frac{1}{2} \leq \limsup_{N \rightarrow \infty} \frac{ S_2 (N) }{N^{1/2} } \leq \frac{1}{2^{1/2} }.
\]
It would be interesting to determine the above limit.  In the case of Sidon sets, we have
$\limsup_{N \rightarrow \infty} \frac{ S_1 (N) }{N^{1/2}} = 1$ by \cite{ET} and \cite{S}.


\end{document}